\documentclass[10pt]{amsart}

\usepackage[a4paper,margin=2.5cm]{geometry}
\usepackage{amsmath,amssymb,amsthm,mathtools,mathrsfs}
\usepackage[T1]{fontenc}
\usepackage[utf8]{inputenc}
\usepackage{lmodern}
\usepackage{microtype}
\usepackage{hyperref}
\usepackage{enumitem}

\hypersetup{
  colorlinks=false,
  linkbordercolor={1 0 0},
  citebordercolor={0 1 0},
  urlbordercolor={0 1 1}
}

\newtheoremstyle{spacedplain}%
  {3pt plus 1pt minus 1pt}%
  {3pt plus 1pt minus 1pt}%
  {\itshape}%
  {}%
  {\bfseries}%
  {.}%
  {.5em}%
  {}
  
\theoremstyle{spacedplain}
\newtheorem{theorem}{Theorem}[section]
\newtheorem{proposition}[theorem]{Proposition}
\newtheorem{lemma}[theorem]{Lemma}
\newtheorem{corollary}[theorem]{Corollary}

\newcommand{\dd}{\,\mathrm{d}}
\newcommand{\E}{\mathbb E}
\newcommand{\Pp}{\mathbb P}

\allowdisplaybreaks[2]

\AtBeginDocument{%
  \setlength{\abovedisplayskip}{6pt plus 2pt minus 2pt}%
  \setlength{\belowdisplayskip}{6pt plus 2pt minus 2pt}%
  \setlength{\abovedisplayshortskip}{3pt plus 2pt minus 1pt}%
  \setlength{\belowdisplayshortskip}{3pt plus 2pt minus 1pt}%
}

\title{Asymptotic Formulae For Reciprocal Partitions}
\author{Nilotpal Kanti Sinha}
\address{Dubai, United Arab Emirates}
\email{nilotpal.sinha@gmail.com}
\date{}

\begin{document}

\begin{abstract}
For integers $1\le m\le n$, let $s(n,m)$ denote the number of $m$-tuples
$(k_1,\ldots,k_m)$ of nonnegative integers satisfying
\[
n=\sum_{j=1}^{m}\frac{k_j}{j}.
\]
We obtain a complete asymptotic expansion for $\log s(n,m)$, uniformly for all
$n\ge m$ as $m\to\infty$. The coefficients are given explicitly in terms of
limiting prime-block functions arising from the residue structure of the
problem. We also determine the full hierarchy of multiplicative corrections in
the sparse regime, identifying explicit constants at every fixed order and, in
particular, the first correction constants $1/4$ and $(1-\log 2)/4$. The
results give uniform two-parameter asymptotics as both the target and the
number of allowed reciprocal parts grow.
\end{abstract}
\maketitle

\section{Introduction and main results}\label{sec:intro}

A natural partition problem arises by decomposing a quantity into pieces whose
allowed sizes are reciprocals of integers.  For a quantity of size $n$ and a
positive integer $m$, allow pieces of sizes
\[
1,\frac12,\frac13,\ldots,\frac1m.
\]
If $k_j$ denotes the number of pieces of size $1/j$, an exact decomposition
satisfies
\begin{equation}\label{eq:def-s}
n=\frac{k_1}{1}+\frac{k_2}{2}+\cdots+\frac{k_m}{m},
\qquad k_1,\ldots,k_m\in\mathbb Z_{\ge0}.
\end{equation}
We write $s(n,m)$ for the number of such decompositions.  Thus $s(n,m)$ counts
partitions of a quantity of size $n$ into the reciprocal parts
$1,1/2,\ldots,1/m$, where only the number of pieces of each size matters.

The diagonal case $m=n$ was posed by Sinha on MathOverflow in
2012 \cite{SinhaMO12}. McKay and Alekseyev provided numerical
computations, with the resulting diagonal sequence recorded as OEIS A208480
\cite{OEIS208480}.  In the same discussion, Meyerowitz introduced the
two-parameter version corresponding to $s(n,m)$ and developed the
quotient--residue decomposition $k_j=jq_j+r_j$ used below.

After multiplication by a common denominator, this problem becomes a
Sylvester denumerant.  Classical denumerant theory starts from a fixed list of
integer coefficients and studies the resulting quasipolynomial counting
function; see Sylvester \cite{Sylvester1857} and Beck and Robins
\cite{BeckRobins15}.  Explicit formulas for polynomial parts and Sylvester
waves were developed further by Dilcher and Vignat \cite{DilcherVignat17},
O'Sullivan \cite{OSullivan16,OSullivan18,OSullivan20}, Kiran \cite{Kiran22},
and Xin and Zhang \cite{XinZhang25}.  There is also a different asymptotic
literature in which restrictions on ordinary integer partitions grow; for
example, Jiang and Wang \cite{JiangWang19} allow the largest part and the
number of parts to vary together.  The phrase ``reciprocal sums of parts'' is
also used for a different statistic: Kim and Kim \cite{KimKim25} study the
sum of the reciprocals of the parts of an ordinary integer partition.  Here,
by contrast, the allowed part sizes themselves are $1,1/2,\ldots,1/m$.

Our setting differs in three ways.  After multiplication by
$L_m=\operatorname{lcm}(1,\ldots,m)$ the coefficient vector is
\[
 L_m,\frac{L_m}{2},\ldots,\frac{L_m}{m}.
\]
Thus (i) the number of coefficients grows, (ii) the coefficients themselves
change with $m$, and (iii) their common period $L_m$ grows exponentially on
the logarithmic scale.  Fixed-coefficient denumerant formulas therefore do
not give a uniform estimate as $m\to\infty$, while the growing-restriction
results for ordinary partitions concern a different generating function and
do not resolve this changing congruence modulus.  The first theorem below is
uniform for every $n\ge m$ and identifies the arithmetic correction to every
fixed order in $1/\log m$.  A second theorem goes beyond logarithmic accuracy
in the sparse range: for every fixed order $R$ it gives explicit
multiplicative corrections uniformly down to the scale
$n\asymp m^2/(\log m)^R$.  To the best of our knowledge, such two-parameter
asymptotics for this reciprocal coefficient vector have not appeared previously.

Put
\[
 L_m=\operatorname{lcm}(1,2,\ldots,m),\qquad
 \psi(m)=\log L_m=\sum_{p^a\le m}\log p.
\]
For fixed $m$, multiplying \eqref{eq:def-s} by $L_m$ gives a classical denumerant.
Here, however, the dimension, coefficient vector, and period all grow.  The useful
structure is instead the prime-power factorization of $L_m$: primes $p>\sqrt m$ act on
disjoint coordinate blocks, whose small biases generate the inverse-logarithmic
corrections.

Write $\alpha=m/n\in(0,1]$.  The main theorem separates the arithmetic scale $m$ from
the moving saddle $\alpha$.

\begin{theorem}[Uniform two-parameter expansion]\label{thm:main}
For every fixed integer $R\ge1$, there are explicitly computable functions
$E_1,\ldots,E_R\in C^\infty([0,1])$ such that, uniformly for integers $n\ge m$ as
$m\to\infty$,
\begin{equation}\label{eq:main-psi}
 \log s(n,m)
 =m\log n-\psi(m)+\log\frac mn
 +m\sum_{r=1}^{R}\frac{E_r(m/n)}{(\log m)^r}
 +o\!\left(\frac{m}{(\log m)^R}\right).
\end{equation}
Equivalently, for every fixed $R$,
\begin{equation}\label{eq:main-simple}
 \log s(n,m)
 =m\log n-m+\log\frac mn
 +m\sum_{r=1}^{R}\frac{E_r(m/n)}{(\log m)^r}
 +o\!\left(\frac{m}{(\log m)^R}\right).
\end{equation}
\end{theorem}

Here and throughout, ``complete asymptotic expansion'' is meant in the
Poincar\'e sense: for each fixed truncation order $R$, the error is smaller
than the last retained scale; no convergence of the resulting infinite formal
series is asserted.  The uniformity in Theorem~\ref{thm:main} means that if
$\mathcal R_R(n,m)$ denotes the difference between the left side of
\eqref{eq:main-psi} and the displayed terms on its right, then
\[
 \sup_{n\ge m}\frac{|\mathcal R_R(n,m)|}{m/(\log m)^R}\longrightarrow0
 \qquad(m\to\infty).
\]

The second form follows from the classical zero-free-region estimate
$\psi(m)=m+o(m/(\log m)^R)$ for every fixed $R$ \cite{MontgomeryVaughan07}.
At the diagonal point, the first three coefficients are
\begin{equation}\label{eq:diag-values}
 \begin{aligned}
 E_1(1)&=0.225872292\ldots,\\
 E_2(1)&=0.046894482\ldots,\\
 E_3(1)&=0.017888502\ldots.
 \end{aligned}
\end{equation}
Hence the original diagonal problem is an immediate specialization.

\begin{corollary}[Diagonal case]\label{cor:diagonal}
For every fixed $R\ge1$,
\[
 \log s(n,n)=n\log n-n
 +n\sum_{r=1}^{R}\frac{E_r(1)}{(\log n)^r}
 +o\!\left(\frac{n}{(\log n)^R}\right).
\]
\end{corollary}

There is finer multiplicative structure when $m/n\to0$.  For $r\ge1$ put
\begin{equation}\label{eq:kappa-r}
 \kappa_r=\frac{(r-1)!}{2}
 \left(1-\frac12\sum_{j=0}^{r-1}\frac{(\log 2)^j}{j!}\right)
 =\frac12\int_0^{\log 2}t^{r-1}e^{-t}\,\dd t.
\end{equation}
The leading term $(m/L_m)n^{m-1}$ then admits an explicit hierarchy of sparse
corrections.

\begin{theorem}[Sparse asymptotics]\label{thm:transition}
Fix an integer $R\ge1$ and $C>0$.  Uniformly for integers
\[
 n\ge \frac{m^2}{C(\log m)^R},
\]
as $m\to\infty$,
\begin{equation}\label{eq:sparse-all-orders}
 \log\frac{s(n,m)}{(m/L_m)n^{m-1}}
 =\frac{m^2}{n}\sum_{r=1}^{R}\frac{\kappa_r}{(\log m)^r}+o(1).
\end{equation}
In particular $\kappa_1=1/4$ and $\kappa_2=(1-\log 2)/4$.  Hence if
$n\log m/m^2\to c\in(0,\infty]$, then
\begin{equation}\label{eq:transition}
 \frac{s(n,m)}{(m/L_m)n^{m-1}}
 \longrightarrow \exp\!\left(\frac1{4c}\right),
\end{equation}
where $1/\infty=0$.  If $n\log m/m^2\to\infty$, this reduces to
$s(n,m)\sim(m/L_m)n^{m-1}$.  At the next sparse scale, for every fixed $c>0$,
\begin{equation}\label{eq:second-sparse-scale}
 n=\left\lfloor c\frac{m^2}{(\log m)^2}\right\rfloor
 \quad\Longrightarrow\quad
 s(n,m)\sim
 e^{(1-\log 2)/(4c)}m^{1/(4c)}\frac{m}{L_m}\,n^{m-1}.
\end{equation}
More generally, for fixed $R$ and $c>0$, if
$n=\lfloor c m^2/(\log m)^R\rfloor$, then
\begin{equation}\label{eq:R-sparse-scale}
 s(n,m)\sim \frac{m}{L_m}n^{m-1}
 \exp\!\left\{\frac1c\sum_{r=1}^{R}
 \kappa_r(\log m)^{R-r}\right\}.
\end{equation}
\end{theorem}

The sparse coefficients are the boundary derivatives of the coefficient functions in
Theorem~\ref{thm:main}: for every fixed $r\ge1$,
\[
 E_r'(0)=A_r'(0)=\kappa_r.
\]
Thus the sparse hierarchy is the small-$m/n$ boundary behavior of the same
two-parameter expansion.  Here, as in Theorem~\ref{thm:main}, the truncation
order $R$ is fixed; no convergence assertion as $R\to\infty$ is intended.

\paragraph{How the proof fits together.}
The proof is easier to follow if one first separates the different roles played by
$m$ and $n$.  The cutoff $m$ determines which reciprocal parts are available and hence
controls the arithmetic of the problem.  The target $n$ determines which coefficient
of the resulting generating function we must extract.

The first step is to write
\[
 k_j=jq_j+r_j,\qquad 0\le r_j<j.
\]
Then $k_j/j=q_j+r_j/j$.  The quotient variables $q_j$ form an ordinary composition
problem and are responsible for the main combinatorial growth, while all of the
number-theoretic difficulty is pushed into the requirement that
$\sum r_j/j$ be an integer.  After multiplication by
$L_m=\operatorname{lcm}(1,\ldots,m)$, this becomes a single congruence modulo $L_m$.
If the residues were perfectly equidistributed modulo $L_m$, the congruence would cost
a factor $L_m^{-1}$; after taking logarithms this explains the term
$-\psi(m)=-\log L_m$ in Theorem~\ref{thm:main}.

The residues are not exactly equidistributed, and the rest of the arithmetic analysis
measures this small bias.  We factor the congruence prime by prime.  For
$p>\sqrt m$, only the coordinates
\[
 p,2p,\ldots,\lfloor m/p\rfloor p
\]
are involved, and different such primes act on disjoint coordinate blocks.  The
normalized bias of one block is $F_{p,M}(z)$.  For fixed block size $M$ these factors
approach explicit limits $F_M(z)$, and the prime number theorem turns the accumulated
small biases into the functions $A_1(z),A_2(z),\ldots$.

The parameter $z$ is an exponential tilt used only to select the desired coefficient.
At first order the correct tilt is
\[
 z\sim \frac mn.
\]
Thus the arithmetic expansion produced by the functions $A_r(z)$ is evaluated at a
moving saddle determined by $m/n$.  Correcting that saddle order by order produces the
functions $E_r(m/n)$ in Theorem~\ref{thm:main}.  In this way the principal term
$m\log n$, the arithmetic penalty $-\psi(m)$, and the inverse-logarithmic corrections
all come from separate and identifiable parts of the argument.

For $m\le n\le m^2$ this moving-saddle method is uniform.  When $n>m^2$, the saddle
moves very close to zero and an exact identity coming from the same quotient--residue
decomposition is simpler:
\[
 s(n,m)=\sum_h c_m(h)\binom{n-h+m-1}{m-1}.
\]
This identity gives the far-sparse range directly, fixes the lower-order term
$\log(m/n)$ uniformly there, and yields the multiplicative sparse asymptotics
in Theorem~\ref{thm:transition}.  At every fixed order the relevant coefficient
comes from the smallest large-prime block, producing the explicit constants
$\kappa_r$ and extending the analysis to $n\asymp m^2/(\log m)^R$.

\section{Residues and exact decomposition}\label{sec:residue}

We begin with the elementary decomposition that separates the easy counting part from
the arithmetic part.  The quotient variables below will later be handled by ordinary
composition counting, while the residues contain exactly the obstruction to the sum
being an integer.

Write uniquely
\begin{equation}\label{eq:q-r}
 k_j=jq_j+r_j,\qquad q_j\ge0,\qquad 0\le r_j<j.
\end{equation}
Then
\[
 \sum_{j\le m}\frac{k_j}{j}=\sum_{j\le m}q_j+\sum_{j\le m}\frac{r_j}{j},
\]
so integrality depends only on the residue sum.  For $h\in\mathbb Z$, let
\[
 c_m(h)=\#\left\{(r_1,\ldots,r_m):0\le r_j<j,
 \ \sum_{j\le m}\frac{r_j}{j}=h\right\}.
\]
Necessarily $0\le h<m$.

\begin{proposition}[Exact binomial decomposition]\label{prop:binomial}
For every $n\ge0$,
\begin{equation}\label{eq:binomial-mixture}
 s(n,m)=\sum_{h=0}^{m-1}c_m(h)\binom{n-h+m-1}{m-1},
\end{equation}
with the usual convention when $n-h<0$.  Moreover,
\begin{equation}\label{eq:cm-total}
 \sum_hc_m(h)=\frac{m!}{L_m}.
\end{equation}
\end{proposition}

\begin{proof}
For a residue vector with integral sum $h$, the quotient variables satisfy
$q_1+\cdots+q_m=n-h$, giving the binomial coefficient in \eqref{eq:binomial-mixture}.
For \eqref{eq:cm-total}, consider
\[
 \Phi_m:\prod_{j\le m}\mathbb Z/j\mathbb Z\longrightarrow\mathbb Z/L_m\mathbb Z,
 \qquad
 \Phi_m(r)=\sum_{j\le m}r_j\frac{L_m}{j}\pmod{L_m}.
\]
This homomorphism is surjective because
$\gcd(L_m/1,\ldots,L_m/m)=1$: for each prime power $p^a\Vert L_m$, some $j\le m$ is
divisible by $p^a$.  Hence every fiber has size $m!/L_m$, and the zero fiber is exactly
the set counted on the left of \eqref{eq:cm-total}.
\end{proof}

For $z\ge0$ set, by continuity at $z=0$,
\[
 S_j(z)=\sum_{r=0}^{j-1}e^{-zr/j}=\frac{1-e^{-z}}{1-e^{-z/j}},\qquad
 \mu_{j,z}(r)=\frac{e^{-zr/j}}{S_j(z)}.
\]
Let $R_1,\ldots,R_m$ be independent with laws $\mu_{j,z}$ and define
\[
 P_m(z)=\Pp_z\!\left(\sum_{j\le m}\frac{R_j}{j}\in\mathbb Z\right).
\]
Proposition~\ref{prop:binomial} gives $P_m(0)=L_m^{-1}$.  For $q\in\mathbb Z_{\ge0}$ define
\[
 f_m(q)=\#\left\{k_j\ge0:q=\sum_{j\le m}\frac{k_j}{j}\right\},\qquad
 G_m(z)=\sum_{q\in\mathbb Z_{\ge0}}f_m(q)e^{-zq}.
\]
Then the quotient--residue decomposition gives
\begin{equation}\label{eq:G-factor}
 G_m(z)=P_m(z)\prod_{j=1}^m\frac1{1-e^{-z/j}}.
\end{equation}
Indeed, the $q_j$ contribute $(1-e^{-z})^{-m}$, while the weighted residue sum contributes
$P_m(z)\prod_jS_j(z)$.

We next factor the residue congruence prime by prime.  For each $p\le m$, put
\[
 Q_p=p^{a_p},\qquad a_p=\max\{a:p^a\le m\},
\]
so $L_m=\prod_{p\le m}Q_p$.  For $p>\sqrt m$, write
$M_p=\lfloor m/p\rfloor$.  The only coordinates divisible by $p$ are
$p,2p,\ldots,M_pp$, and reduction modulo $p$ gives
\begin{equation}\label{eq:Ep}
 E_p:\qquad \sum_{a=1}^{M_p}\overline a\,R_{ap}\equiv0\pmod p,
\end{equation}
where $\overline a$ denotes the inverse of $a$ modulo $p$.
Distinct primes $p>\sqrt m$ act on disjoint coordinate sets.  The remaining
prime-power congruences can be imposed through one control coordinate for each
small prime.

\begin{lemma}[Small-prime controls]\label{lem:controls}
Let $h_m=\pi(\sqrt m)$.  Fix all residues except $R_{Q_p}$ for
$p\le\sqrt m$.  There is a unique choice of the control vector
$(R_{Q_p})_{p\le\sqrt m}$ that makes the full residue sum integral.  Moreover,
for every $z_1>0$ there are constants $0<c_{z_1}<C_{z_1}<\infty$ such that
\begin{equation}\label{eq:control-weight}
 c_{z_1}\le Q\mu_{Q,z}(r)\le C_{z_1}
 \qquad(0\le z\le z_1,\ 0\le r<Q).
\end{equation}
If $\Pp_z^{>}$ denotes the product residue law conditioned on all large-prime
events $E_q$, $q>\sqrt m$, and projected onto the noncontrol coordinates, and
$\Pp_z^{\rm all}$ is the corresponding law after imposing the full congruence,
then
\begin{equation}\label{eq:RN-controls}
 e^{-C_{z_1}h_m}\le
 \frac{\dd\Pp_z^{\rm all}}{\dd\Pp_z^{>}}
 \le e^{C_{z_1}h_m}.
\end{equation}
Furthermore,
\begin{equation}\label{eq:control-prob}
 \Pp_z\!\left(\text{full congruence}\mid E_q\text{ for all }q>\sqrt m\right)
 =\left(\prod_{p\le\sqrt m}\frac1{Q_p}\right)e^{O_{z_1}(h_m)}.
\end{equation}
\end{lemma}

\begin{proof}
Modulo $Q_p$, the coefficient $L_m/Q_p$ of the control $R_{Q_p}$ is
invertible.  If $q\ne p$, the coefficient $L_m/Q_q$ of the other pure
prime-power control is divisible by $Q_p$.  Hence the congruence modulo $Q_p$
determines $R_{Q_p}$ uniquely, independently of the other controls; the Chinese
remainder theorem gives the simultaneous choice.  Changing such a small-prime
control does not disturb any large-prime condition $E_q$, because $L_m/Q_p$ is
divisible by every prime $q>\sqrt m$.

The identity
\[
 Q\mu_{Q,z}(r)=
 \frac{Q(1-e^{-z/Q})}{1-e^{-z}}e^{-zr/Q}
\]
(with its continuous value $1$ at $z=0$) proves \eqref{eq:control-weight}
uniformly on $0\le z\le z_1$.

Now fix a noncontrol configuration $X$ satisfying all large-prime conditions,
and let $r^*(X)$ be its unique completing control vector.  Its conditional
weight is
\[
 W_z(X)=\prod_{p\le\sqrt m}\mu_{Q_p,z}(r^*_{Q_p}(X)).
\]
By \eqref{eq:control-weight},
\[
 \left(\prod_{p\le\sqrt m}\frac1{Q_p}\right)c^{h_m}
 \le W_z(X)\le
 \left(\prod_{p\le\sqrt m}\frac1{Q_p}\right)C^{h_m}.
\]
Taking expectation under $\Pp_z^{>}$ gives \eqref{eq:control-prob}, while
Bayes' formula
\[
 \frac{\dd\Pp_z^{\rm all}}{\dd\Pp_z^{>}}(X)
 =\frac{W_z(X)}{\E_z^{>}W_z(X)}
\]
gives \eqref{eq:RN-controls}.
\end{proof}

Since the large-prime events are independent, Lemma~\ref{lem:controls} gives
\begin{equation}\label{eq:P-block}
 \log P_m(z)
 =-\psi(m)+\sum_{\sqrt m<p\le m}\log F_{p,M_p}(z)
 +O_{z_1}(h_m),
\end{equation}
where
\[
 F_{p,M}(z)=p\,\Pp_z(E_p).
\]
The Radon--Nikodym comparison \eqref{eq:RN-controls} will be used again in the
coefficient localization below.

\section{Large-prime blocks}\label{sec:blocks}

We now measure how far the residue congruence is from perfect uniformity.  Large primes
are especially useful because their coordinate sets are disjoint, so their effects can
be analyzed independently and then added on the logarithmic scale.  These blocks contain
all corrections visible on the scale $m/(\log m)^R$.  Because the final expansion is
claimed to every fixed order, we keep the derivative bounds below explicit rather than
appeal to a formal differentiation argument.

\begin{lemma}[Fourier formula]\label{lem:Fourier}
For $p>M$,
\begin{equation}\label{eq:FpM}
 F_{p,M}(z)=\sum_{b=0}^{p-1}\prod_{a=1}^{M}
 \frac{1-e^{-z/(ap)}}{1-e^{-z/(ap)}e^{2\pi i b\overline a/p}}.
\end{equation}
\end{lemma}

\begin{proof}
Apply the finite Fourier filter to \eqref{eq:Ep}.  For the coordinate $ap$, the normalized
characteristic sum is
\[
 \frac1{S_{ap}(z)}\sum_{r=0}^{ap-1}e^{-zr/(ap)}e^{2\pi i b\overline a r/p}
 =\frac{1-e^{-z/(ap)}}{1-e^{-z/(ap)}e^{2\pi i b\overline a/p}},
\]
because the numerator of the geometric sum is $1-e^{-z}$.
\end{proof}

Let $D_M=\operatorname{lcm}(1,\ldots,M)$.

\begin{proposition}[Block limits and uniform estimates]\label{prop:block}
For fixed $M$, as $p\to\infty$ through primes,
$F_{p,M}(z)\to F_M(z)$ locally uniformly for $z\ge0$, where
\begin{equation}\label{eq:FM}
 F_1(z)=\frac{z}{1-e^{-z}},\qquad
 F_M(z)=\sum_{\ell\in\mathbb Z}
 \left(\frac{z}{z+2\pi i\ell D_M}\right)^M\quad(M\ge2),
\end{equation}
with $F_M(0)=1$.  Each $F_M$ is positive.

More precisely, fix $z_1>0$ and an integer $J\ge0$.  There are constants
$C_J,B_J<\infty$ and $0<\rho<1$, depending only on $J$ and $z_1$, such that
for $M\ge2$ and $0\le j\le J$,
\begin{equation}\label{eq:block-decay}
 \sup_{0\le z\le z_1}|\partial_z^j\log F_M(z)|
 \le C_JM^{B_J}\rho^M.
\end{equation}
Every assertion below involving the finite block $F_{p,M}$ is understood for
primes $p>M$.  In that range the same estimate holds for $F_{p,M}$.  If, in
addition, $M\le p^{1/4}$, then
\begin{equation}\label{eq:block-quant}
 \sup_{0\le z\le z_1}|\partial_z^j(F_{p,M}-F_M)(z)|
 \le C_JM^{B_J}\left(\frac{\log p}{p}+\frac1{\sqrt p}\right).
\end{equation}
For $M=1$, the corresponding differences are $O_{z_1,J}(p^{-1})$.
\end{proposition}

\begin{proof}
We give the details because the estimates below are used after arbitrarily
many, but fixed, differentiations with respect to $z$.

Write
\[
 R_{ap}=au+v,\qquad 0\le u<p,\quad 0\le v<a.
\]
The factor belonging to the coordinate $ap$ in
Lemma~\ref{lem:Fourier} splits as
\[
 A_p(b)B_{p,a}(b),
\]
where
\begin{equation}\label{eq:Ap}
 A_p(b)=\frac{1-e^{-z/p}}{1-e^{(-z+2\pi ib)/p}}
\end{equation}
and
\begin{equation}\label{eq:Bpa}
 B_{p,a}(b)=
 \frac{\displaystyle\sum_{v=0}^{a-1}e^{-zv/(ap)}
 e^{2\pi ib\overline a v/p}}
 {\displaystyle\sum_{v=0}^{a-1}e^{-zv/(ap)}}.
\end{equation}
We always take the centered representative $|b|\le(p-1)/2$.

We first record two elementary bounds.  They are the only estimates on the
individual factors that will be needed.  Let
\[
 a_b(z)=\frac{z}{z-2\pi ib}\qquad (b\ne0).
\]
For every fixed $J$ and $z_1$ there are constants $C_J$ and $0<\rho<1$
such that, for $b\ne0$, $0\le z\le z_1$, and $0\le j\le J$,
\begin{equation}\label{eq:A-detailed}
 |\partial_z^j A_p(b)|\le \frac{C_J}{1+|b|},
 \qquad |A_p(b)|\le\rho,
\end{equation}
and
\begin{equation}\label{eq:B-detailed}
 |B_{p,a}(b)|\le1,\qquad
 |\partial_z^jB_{p,a}(b)|\le C_J\quad(1\le j\le J).
\end{equation}
To see the first estimate, put
$\phi(w)=(1-e^{-w})/w$, with $\phi(0)=1$.  Then
\begin{equation}\label{eq:A-phi}
 A_p(b)=a_b(z)\,
 \frac{\phi(z/p)}{\phi((z-2\pi ib)/p)}.
\end{equation}
For centered $b$, the second argument of $\phi$ stays in the compact strip
$0\le\Re w\le z_1/2$, $|\Im w|\le\pi$ (the finitely many smaller $p$ are
absorbed into the constants).  The function $\phi$ has no zero there.
Differentiating \eqref{eq:A-phi} therefore gives the first inequality in
\eqref{eq:A-detailed}.  The uniform contraction follows directly from
\[
 |A_p(b)|^2=
 \frac{(1-e^{-z/p})^2}
 {(1-e^{-z/p})^2+2e^{-z/p}(1-\cos(2\pi b/p))}.
\]
For $b\ne0$, the denominator contains a positive term of size at least a
constant multiple of $p^{-2}$, whereas
$1-e^{-z/p}\le z_1/p$.  Hence the right side is bounded by a constant
strictly smaller than $1$, depending only on $z_1$.

For \eqref{eq:B-detailed}, note first that $B_{p,a}(b)$ is a weighted average
of numbers of modulus one, so $|B_{p,a}(b)|\le1$.  If
$x_v=v/(ap)$, then $0\le x_v<1/p$ and the numerator and denominator in
\eqref{eq:Bpa} are sums of $e^{-zx_v}$, with or without a unit-modulus
factor.  Their $j$th derivatives are bounded by $a p^{-j}$.  The denominator
is at least $ae^{-z_1/p}$.  Repeated use of the quotient rule now proves the
second estimate in \eqref{eq:B-detailed}; in fact one obtains the stronger
bound $O_J(p^{-j})$ for $j\ge1$.

We next identify the limit of a fixed Fourier mode.  If
$a\overline a=1+\nu_ap$, then $\gcd(\nu_a,a)=1$ and
\[
 e^{2\pi ib\overline a v/p}
 =e^{2\pi ibv/(ap)}e^{2\pi ib\nu_av/a}.
\]
For fixed $a$ and $b$, the first factor tends to $1$.  The average of the
second factor over $0\le v<a$ equals $1$ if $a\mid b$ and $0$ otherwise.
Thus
\[
 B_{p,a}(b)\longrightarrow \mathbf 1_{a\mid b}.
\]
Together with $A_p(b)\to a_b(z)$, a mode survives precisely when every
$a\le M$ divides $b$, or equivalently when $D_M\mid b$.  The mode $b=0$
is identically $1$.  This proves the formula \eqref{eq:FM} once the tails are
controlled; that control is included below and also gives local uniform
convergence.  The positivity of the limiting expression can also be seen
directly.  For $z>0$, Poisson summation applied to the gamma density
$z^Mx^{M-1}e^{-zx}/(M-1)!$ on $x>0$ gives
\begin{equation}\label{eq:FM-positive}
 F_M(z)=\frac{z^M}{D_M^M(M-1)!}
 \sum_{r=1}^{\infty}r^{M-1}e^{-zr/D_M}>0
 \qquad(M\ge2).
\end{equation}
At $z=0$ positivity follows by continuity from $F_M(0)=1$.

We now prove exponential decay in the block size.  For a nonzero mode write
\[
 U_{p,M}(b;z)=A_p(b)^M\prod_{a\le M}B_{p,a}(b).
\]
Apply Leibniz' rule to $\partial_z^jU_{p,M}$.  There are at most
$(2M)^j$ terms.  In any such term at most $j$ of the $M$ copies of $A_p$
are differentiated.  Every differentiated copy is
$O_J((1+|b|)^{-1})$ by \eqref{eq:A-detailed}; every undifferentiated copy is
both $O((1+|b|)^{-1})$ and at most $\rho$ in modulus.  Using two copies for
summability in $b$ and all remaining undifferentiated copies for the
contraction, we obtain, after changing the constants,
\begin{equation}\label{eq:mode-decay}
 |\partial_z^jU_{p,M}(b;z)|
 \le C_JM^{B_J}\rho^{M-O_J(1)}(1+|b|)^{-2}.
\end{equation}
When $M$ is bounded in terms of $J$, the same conclusion follows after
increasing $C_J$.  Summing \eqref{eq:mode-decay} over $b\ne0$ gives
\begin{equation}\label{eq:Fminus1-decay}
 \sup_{0\le z\le z_1}
 |\partial_z^j(F_{p,M}(z)-1)|
 \le C_JM^{B_J}\rho^M.
\end{equation}
The limiting series is handled in exactly the same way, with nonzero modes
$b=D_M\ell$.  This also justifies passage to the limit mode by mode and shows
local uniform convergence, with derivatives of every fixed order.

Since $F_{p,M}(z)=p\,\Pp_z(E_p)>0$, every finite block is positive, and
\eqref{eq:FM-positive} proves positivity of the limiting block.  Therefore
the logarithms are well defined.  For large $M$ both
$F_{p,M}$ and $F_M$ lie, say, in $[1/2,3/2]$, so differentiating
$\log F$ expresses its derivatives as finite sums of products of derivatives
of $F-1$.  The finitely many smaller $M$ are again absorbed into the
constant.  This proves \eqref{eq:block-decay}.

It remains to quantify the convergence $F_{p,M}\to F_M$.  Put
$H=\lfloor\sqrt p\rfloor$.  We first consider $|b|\le H$.  From
\eqref{eq:A-phi} and Taylor's theorem for $\phi$ at the origin, uniformly for
$j\le J$,
\begin{equation}\label{eq:A-low}
 \left|\partial_z^j\bigl(A_p(b)-a_b(z)\bigr)\right|
 \le \frac{C_J}{p}.
\end{equation}
For the second factor, use again
$a\overline a=1+\nu_ap$ and write
\[
 e^{-zv/(ap)}e^{2\pi ib\overline a v/p}
 =e^{(-z+2\pi ib)v/(ap)}e^{2\pi ib\nu_av/a}.
\]
Because $v/a\le1$ and $|b|\le H$, the first exponential is
$1+O_{z_1}((1+|b|)/p)$, uniformly in $a$.  The average of the second
exponential is $\mathbf1_{a\mid b}$.  Differentiating only inserts powers of
$v/(ap)\le1/p$.  Hence
\begin{equation}\label{eq:B-low}
 \left|\partial_z^j\bigl(B_{p,a}(b)-\mathbf1_{a\mid b}\bigr)\right|
 \le C_J\frac{1+|b|}{p}
 \qquad (0\le j\le J).
\end{equation}

Define the limiting mode
\[
 U_M(b;z)=\mathbf1_{D_M\mid b}\,a_b(z)^M
\]
for $b\ne0$, and set both zero modes equal to $1$.  We spell out the
comparison, since this is where the dependence on $M$ and $b$ enters the
quantitative error.  Regard $U_{p,M}$ as a product of $2M$ factors,
consisting of $M$ copies of $A_p(b)$ followed by the factors $B_{p,a}(b)$,
and regard $U_M$ as the corresponding product of $M$ copies of $a_b(z)$
followed by the constants $\mathbf1_{a\mid b}$.  For any two lists
$(X_r)_{r\le2M}$ and $(Y_r)_{r\le2M}$ we use the exact telescoping identity
\begin{equation}\label{eq:product-telescope}
 \prod_{r=1}^{2M}X_r-\prod_{r=1}^{2M}Y_r
 =\sum_{s=1}^{2M}(X_s-Y_s)
   \prod_{r<s}X_r\prod_{r>s}Y_r.
\end{equation}
After taking $j$ derivatives, Leibniz' rule produces at most
$C_JM^{J+1}$ terms: there are $2M$ choices for the telescoping position and,
for each one, at most $(2M)^j$ ways to distribute the derivatives.  The same
bounds used above also give
\[
 |\partial_z^q a_b(z)|\le \frac{C_J}{1+|b|}
 \qquad(0\le q\le J).
\]
If the distinguished difference in \eqref{eq:product-telescope} is an
$A$-factor, \eqref{eq:A-low} contributes $O_J(p^{-1})$, while the remaining
$M-1$ $A$-type factors, differentiated or not, contribute
$O_J((1+|b|)^{1-M})$.  All $B$-type factors and their derivatives are
$O_J(1)$.  If the distinguished difference is a $B$-factor, then
\eqref{eq:B-low} contributes $O_J((1+|b|)/p)$, while all $M$ $A$-type
factors contribute $O_J((1+|b|)^{-M})$.  Thus both kinds of telescoping terms
have the same bound.  Absorbing the polynomial number of Leibniz terms into
$M^{B_J}$ gives
\begin{equation}\label{eq:low-mode-diff}
 |\partial_z^j(U_{p,M}(b;z)-U_M(b;z))|
 \le \frac{C_JM^{B_J}}p(1+|b|)^{1-M}
\end{equation}
for $1\le |b|\le H$.  For $M=2$ the sum over $b$ is
$O((\log H)/p)=O((\log p)/p)$; for $M\ge3$ the series
$\sum_{b\ne0}(1+|b|)^{1-M}$ is uniformly bounded.  Therefore
\begin{equation}\label{eq:low-sum}
 \sum_{1\le|b|\le H}
 |\partial_z^j(U_{p,M}(b;z)-U_M(b;z))|
 \le C_JM^{B_J}\frac{\log p}{p}.
\end{equation}

For the remaining frequencies we do not compare modes individually.
The bound \eqref{eq:mode-decay}, with two powers reserved for summability,
gives
\[
 \sum_{H<|b|\le(p-1)/2}|\partial_z^jU_{p,M}(b;z)|
 \le C_JM^{B_J}H^{-1}.
\]
The same estimate holds for the tail of the limiting series, since its
nonzero modes are a subset of the nonzero integers.  As $H\asymp\sqrt p$,
combining this with \eqref{eq:low-sum} proves \eqref{eq:block-quant}.  The
assumption $M\le p^{1/4}$ is more than enough to keep the polynomial factors
in $M$ harmless in the Taylor estimates above.  Finally, for $M=1$ the
identity
\[
 F_{p,1}(z)=p\frac{1-e^{-z/p}}{1-e^{-z}}
\]
gives the stated $O_{z_1,J}(p^{-1})$ estimate directly.
\end{proof}

For $r\ge1$ define
\[
 P_r(x)=(r-1)!\sum_{j=0}^{r-1}\frac{x^j}{j!},\qquad
 a_{r,M}=\frac{P_r(\log M)}M-\frac{P_r(\log(M+1))}{M+1},
\]
and
\begin{equation}\label{eq:Ar}
 A_r(z)=\sum_{M=1}^{\infty}a_{r,M}\log F_M(z).
\end{equation}
Proposition~\ref{prop:block} makes these series, and every fixed number of derivatives,
absolutely and uniformly convergent on bounded intervals.  In particular $A_r(0)=0$.

\begin{theorem}[Prime-block expansion]\label{thm:prime-block}
Fix integers $R\ge1$ and $J\ge0$.  Uniformly for $0\le z\le z_1$ and
$0\le j\le J$,
\begin{equation}\label{eq:prime-block-exp}
 \sum_{\sqrt m<p\le m}\partial_z^j\log F_{p,\lfloor m/p\rfloor}(z)
 =m\sum_{r=1}^{R}\frac{A_r^{(j)}(z)}{(\log m)^r}
 +o\!\left(\frac{m}{(\log m)^R}\right).
\end{equation}
\end{theorem}

\begin{proof}
Let $L=\log m$ and $K=\lfloor C\log L\rfloor$, with $C$ large in terms of $R$.
By \eqref{eq:block-decay}, blocks with $\lfloor m/p\rfloor>K$ contribute $o(m/L^R)$.
For $M\le K$, all relevant primes satisfy $p\ge m/(K+1)$, so
Proposition~\ref{prop:block}, used with the fixed derivative order $J$, permits
replacement of $F_{p,M}$ by $F_M$, with total error $o(m/L^R)$ uniformly for
$0\le j\le J$.  The prime number theorem with the de la Vall\'ee-Poussin
remainder gives
\[
 \pi(m/M)=\frac mM\sum_{r=1}^{R}\frac{P_r(\log M)}{L^r}
 +O_R\!\left(\frac{m(1+\log M)^R}{ML^{R+1}}\right).
\]
Taking the difference between $M$ and $M+1$, multiplying by
$\partial_z^j\log F_M(z)$, and summing proves \eqref{eq:prime-block-exp}; exponential
decay permits extension of the $M$-sum to infinity.
\end{proof}

Combining \eqref{eq:P-block} and Theorem~\ref{thm:prime-block} gives, uniformly on bounded
$z$-intervals,
\begin{equation}\label{eq:P-exp}
 \log P_m(z)=-\psi(m)+m\sum_{r=1}^{R}\frac{A_r(z)}{(\log m)^r}
 +o\!\left(\frac{m}{(\log m)^R}\right).
\end{equation}
By \eqref{eq:G-factor} and
\[
 -\log(1-e^{-z/j})=\log j-\log z+\frac{z}{2j}+O_{z_1}(j^{-2}),
\]
where the remainder is uniform for every $0<z\le z_1$, we obtain, uniformly
for $(2m)^{-1}\le z\le z_1$,
\begin{equation}\label{eq:G-exp}
 \log G_m(z)=\log m!-\psi(m)-m\log z
 +m\sum_{r=1}^{R}\frac{A_r(z)}{(\log m)^r}
 +o\!\left(\frac{m}{(\log m)^R}\right).
\end{equation}

\section{Moving saddle and proof of the main theorem}\label{sec:saddle}

The preceding section describes the arithmetic for a general exponential tilt $z$.
We now choose $z$ so that, under the tilted counting measure, totals near the desired
value $n$ are typical rather than rare.  This is the role of the saddle point.  To
first order it is $z=m/n$; the formal series below computes its successive corrections.

Put $t=1/\log m$ and write formally
$\mathcal A(z,t)=\sum_{r\ge1}A_r(z)t^r$.  For $\alpha\in[0,1]$, the formal
implicit-function theorem gives a unique series
\[
 W_\alpha(t)=1+\sum_{r\ge1}b_r(\alpha)t^r
\]
satisfying
\begin{equation}\label{eq:saddle-formal}
 1-\frac1{W_\alpha(t)}+\alpha\,\partial_z\mathcal A
 (\alpha W_\alpha(t),t)=0.
\end{equation}
At $t=0$ the solution is $W=1$, and the derivative of the left side with
respect to $W$ is $1$.  Define $E_r(\alpha)$ by
\begin{equation}\label{eq:Er-def}
 W_\alpha(t)-\log W_\alpha(t)+\mathcal A(\alpha W_\alpha(t),t)
 =1+\sum_{r\ge1}E_r(\alpha)t^r.
\end{equation}
The next lemma records the quantitative statement that will actually be used.

\begin{lemma}[Finite-order saddle expansion]\label{lem:finite-saddle}
Fix $R\ge1$, and let
\[
 W_{\alpha,R}(t)=1+\sum_{r=1}^{R}b_r(\alpha)t^r.
\]
Uniformly for $0\le\alpha\le1$ as $t\downarrow0$,
\begin{equation}\label{eq:saddle-residual}
 1-\frac1{W_{\alpha,R}(t)}
 +\alpha\sum_{r=1}^{R}A_r'(\alpha W_{\alpha,R}(t))t^r
 =O_R(t^{R+1}),
\end{equation}
and
\begin{equation}\label{eq:saddle-value-residual}
 W_{\alpha,R}(t)-\log W_{\alpha,R}(t)
 +\sum_{r=1}^{R}A_r(\alpha W_{\alpha,R}(t))t^r
 =1+\sum_{r=1}^{R}E_r(\alpha)t^r+O_R(t^{R+1}).
\end{equation}
All coefficient functions $b_r$ and $E_r$ are smooth on $[0,1]$.  Moreover,
for every fixed $r\ge1$,
\begin{equation}\label{eq:small-alpha-coeff}
 b_r(\alpha)=O_r(\alpha),\qquad E_r(\alpha)=O_r(\alpha)
 \qquad(\alpha\downarrow0).
\end{equation}
\end{lemma}

\begin{proof}
The coefficients $b_r(\alpha)$ are chosen recursively so that the
coefficients of $t,t^2,\ldots$ in \eqref{eq:saddle-formal} vanish.  By
Theorem~\ref{thm:prime-block} and Proposition~\ref{prop:block}, every fixed
number of derivatives of the $A_r$ is bounded on $[0,2]$.  Taylor's theorem
therefore has a remainder uniform in $\alpha\in[0,1]$.  Substituting the
truncated series and keeping terms through degree $R$ gives
\eqref{eq:saddle-residual}.  The same argument applied to
\eqref{eq:Er-def} gives \eqref{eq:saddle-value-residual}.  The recursion uses
only addition, multiplication, and derivatives of the smooth functions
$A_r$, so the coefficient functions are smooth.

We also record the behavior at $\alpha=0$, because it is used when the bulk
and sparse ranges are joined.  Write $W_\alpha(t)=1+B_\alpha(t)$.  The
coefficient of $t^r$ in $1-1/W_\alpha$ is $b_r(\alpha)$ plus a polynomial in
$b_1(\alpha),\ldots,b_{r-1}(\alpha)$ with no constant or linear term.  The
other term in \eqref{eq:saddle-formal} has the explicit prefactor $\alpha$.
Since the derivatives of every fixed $A_j$ are bounded near the origin, an
induction on $r$ gives $b_r(\alpha)=O_r(\alpha)$.  Finally,
$W_\alpha-\log W_\alpha-1=O(B_\alpha^2)$ coefficientwise, while
$A_j(0)=0$ and hence $A_j(\alpha W_\alpha)=O_j(\alpha)$.  Comparing the
coefficient of $t^r$ in \eqref{eq:Er-def} therefore gives
$E_r(\alpha)=O_r(\alpha)$.
\end{proof}

The first three coefficient functions are
\begin{align}
 E_1(\alpha)&=A_1(\alpha),\label{eq:E1}\\
 E_2(\alpha)&=A_2(\alpha)-\frac{\alpha^2}{2}A_1'(\alpha)^2,\label{eq:E2}\\
 E_3(\alpha)&=A_3(\alpha)+\frac{\alpha^3}{3}A_1'(\alpha)^3
 +\frac{\alpha^4}{2}A_1'(\alpha)^2A_1''(\alpha)
 -\alpha^2A_1'(\alpha)A_2'(\alpha).\label{eq:E3}
\end{align}
The estimate $E_r(\alpha)=O_r(\alpha)$ for every fixed $r$ is part of
Lemma~\ref{lem:finite-saddle}.  The diagonal values in \eqref{eq:diag-values} are obtained by evaluating
\eqref{eq:E1}--\eqref{eq:E3} at $\alpha=1$.  A short numerical check is given
at the end of the paper; none of the numerical values is used in the proof.

The coefficient localization must be performed before the small-prime controls are inserted.
From this point onward we enlarge the previously defined residue laws
$\Pp_z^{>}$ and $\Pp_z^{\rm all}$ by adjoining independent quotient variables
$q_j$, each geometric with parameter $e^{-z}$.  Under the enlarged completed
law, the full total $T$ has distribution
$\Pp_z(T=q)=f_m(q)e^{-zq}/G_m(z)$.  Under the enlarged law $\Pp_z^{>}$ put
\begin{equation}\label{eq:Y}
 Y=\sum_{j=1}^{m}q_j+
 \sum_{\substack{j\le m\\j\ne Q_p\ (p\le\sqrt m)}}\frac{R_j}{j}.
\end{equation}
The quotient variables remain independent geometric variables.

\begin{lemma}[Partial mean and variance]\label{lem:mean-var}
For every fixed $R\ge1$, uniformly for $(2m)^{-1}\le z\le z_1$,
\begin{equation}\label{eq:Ymean}
 \E_z^{>}Y=\frac mz-m\sum_{r=1}^{R}\frac{A_r'(z)}{(\log m)^r}
 +o\!\left(\frac{m}{(\log m)^R}\right)+O(h_m+\log m),
\end{equation}
and there are constants $0<c_{z_1}<C_{z_1}<\infty$ such that
\begin{equation}\label{eq:Yvar}
 c_{z_1}\frac{m}{z^2}
 \le \operatorname{Var}_z^{>}(Y)
 \le C_{z_1}\frac{m}{z^2}.
\end{equation}
\end{lemma}

\begin{proof}
Before imposing congruences,
\[
 \sum_{j\le m}\E\frac{k_j}{j}
 =\sum_{j\le m}\frac1{j(e^{z/j}-1)}=\frac mz+O(\log m).
\]
The last estimate is uniform for $0<z\le z_1$: it follows from
$1/(e^u-1)=u^{-1}-1/2+O_{z_1}(u)$ with $u=z/j$.  Thus allowing
$z$ to decrease from $1/m$ to $1/(2m)$ introduces no new error.
Conditioning a large-prime block changes its mean by
$-\partial_z\log F_{p,M_p}(z)$.  Deleting the $h_m$ control residues changes the mean by
$O(h_m)$, so \eqref{eq:Ymean} follows from Theorem~\ref{thm:prime-block}.

The quotient variables are independent geometric variables with parameter
$e^{-z}$.  Hence, uniformly for $0<z\le z_1$,
\[
 \operatorname{Var}(q_j)=\frac{e^{-z}}{(1-e^{-z})^2}\asymp_{z_1}z^{-2}.
\]
They are independent of all residue conditions.  Their contribution alone
therefore gives the lower bound in \eqref{eq:Yvar}, and their total variance
is $O(mz^{-2})$.  It remains to bound the residue variance, which can be read
directly from the block normalizing factors.  For one large-prime block put
\[
 U_p=\sum_{a\le M_p}\frac{R_{ap}}{ap}.
\]
Before conditioning, the variables in this sum are independent and each lies
in $[0,1)$, so $\operatorname{Var}(U_p)=O(M_p)$.  The logarithm of the
normalizing factor after conditioning on $E_p$ differs from the unconditioned
one by $\log F_{p,M_p}(z)$.  Differentiating twice therefore gives the exact
identity
\[
 \operatorname{Var}_z(U_p\mid E_p)
 =\sum_{a\le M_p}\operatorname{Var}_z\!\left(\frac{R_{ap}}{ap}\right)
 +\partial_z^2\log F_{p,M_p}(z).
\]
Proposition~\ref{prop:block}, with $J=2$, shows that the last term is
$O(M_p)$ after enlarging the constant for the finitely many small block
sizes.  Distinct large-prime blocks are disjoint, and all residue coordinates
outside those blocks remain independent.  Since
$\sum_{\sqrt m<p\le m}M_p\le m$, the total residue variance is $O(m)$.
Since $z\le z_1$, this $O(m)$ residue contribution is absorbed by
$O(mz^{-2})$.  Together with the quotient variables this proves both bounds
in \eqref{eq:Yvar}.
\end{proof}

\begin{lemma}[From a probability window to one coefficient]\label{lem:coefficient-window}
For every integer $q\ge0$,
\[
 f_m(q+1)\ge f_m(q).
\]
Moreover, under the completed tilted law described above, if $W\ge0$ and
\[
 \Pp_z(n-W\le T\le n)\ge e^{-H},
\]
then
\begin{equation}\label{eq:coefficient-window}
 \log f_m(n)\ge\log G_m(z)+zn-zW-H-\log(W+1).
\end{equation}
\end{lemma}

\begin{proof}
Adding one to $k_1$ sends every representation of $q$ to a distinct
representation of $q+1$, proving monotonicity.  In the interval
$n-W\le q\le n$ we have $f_m(q)\le f_m(n)$ and
$e^{-zq}\le e^{-z(n-W)}$.  Hence
\[
 e^{-H}\le\Pp_z(n-W\le T\le n)
 \le\frac{(W+1)f_m(n)e^{-z(n-W)}}{G_m(z)}.
\]
Rearranging gives \eqref{eq:coefficient-window}.
\end{proof}

\begin{lemma}[Transferring a window through the controls]\label{lem:window-transfer}
Let $A$ be an event determined by the noncontrol coordinates, and suppose
that on $A$ the partial total $Y$ lies in an interval $[u,v]$.  After the
small-prime controls are completed, let $T$ be the full integral total.  Then
$T-Y\in[0,h_m)$ and, uniformly for $0\le z\le z_1$,
\begin{equation}\label{eq:window-transfer}
 \Pp_z^{\rm all}(A)\ge e^{-C_{z_1}h_m}\Pp_z^{>}(A).
\end{equation}
Consequently, if $Y\in[n-h_m-2w,n-h_m]$ on $A$, then
$T\in[n-h_m-2w,n]$ on $A$ and
\[
 \Pp_z^{\rm all}(n-h_m-2w\le T\le n)
 \ge e^{-C_{z_1}h_m}\Pp_z^{>}(A).
\]
\end{lemma}

\begin{proof}
The completed control vector is a deterministic function of the noncontrol
coordinates.  Each control contributes $R_{Q_p}/Q_p\in[0,1)$, and there are
$h_m$ controls, so $0\le T-Y<h_m$.  The probability comparison is exactly
the lower Radon--Nikodym bound in \eqref{eq:RN-controls}, integrated over
$A$.
\end{proof}

\begin{proposition}[Bulk coefficient extraction]\label{prop:bulk}
Fix $R\ge1$.  Uniformly for $m\le n\le m^2$,
\begin{equation}\label{eq:bulk}
 \log s(n,m)=m\log n-\psi(m)
 +m\sum_{r=1}^{R}\frac{E_r(m/n)}{(\log m)^r}
 +o\!\left(\frac{m}{(\log m)^R}\right).
\end{equation}
\end{proposition}

\begin{proof}
Set $\alpha=m/n\in[m^{-1},1]$, $L=\log m$, and let $W_{\alpha,R}$ be the truncation of
$W_\alpha$ through order $R$.  Put
$z_0=\alpha W_{\alpha,R}(1/L)$.  To see explicitly why this centers the
tilt, substitute $z=\alpha W$ into \eqref{eq:Ymean}.  Up to the stated
remainder there,
\[
 \E_z^{>}Y-n
 =\frac{m}{\alpha}
 \left\{\frac1W-1
 -\alpha\sum_{r=1}^{R}\frac{A_r'(\alpha W)}{L^r}\right\}.
\]
The expression in braces is the negative of the left side of
\eqref{eq:saddle-residual}, with $t=1/L$.  Lemma~\ref{lem:finite-saddle} and
Lemma~\ref{lem:mean-var} therefore give
\[
 \E_{z_0}^{>}Y=n+o\!\left(\frac{m}{\alpha L^R}\right).
\]
Let $w=m^{2/3}/\alpha$.  Since $W_{\alpha,R}(1/L)=1+O_R(L^{-1})$
uniformly for $0\le\alpha\le1$, we have $z_0/\alpha=1+O_R(L^{-1})$.
Hence, for every fixed $\varepsilon>0$ and all sufficiently large $m$, every
$z$ between $z_0-\varepsilon\alpha L^{-R}$ and
$z_0+\varepsilon\alpha L^{-R}$ satisfies, say,
\[
 \frac12\alpha\le z\le2\alpha.
\]
Apply \eqref{eq:G-exp} and Lemma~\ref{lem:mean-var} with the fixed
choice $z_1=2$.  In particular this whole interval lies in their range
$(2m)^{-1}\le z\le2$, and the lower bound in \eqref{eq:Yvar} is uniformly
$\gg m/\alpha^2$ throughout it.  Since
$\dd\E_z^{>}Y/\dd z=-\operatorname{Var}_z^{>}(Y)$, integration over an
interval of length $\varepsilon\alpha L^{-R}$ changes the mean by at least
a constant multiple of $\varepsilon m/(\alpha L^R)$.  As
$h_m+w=o(m/(\alpha L^R))$, the two values
$z_0\pm\varepsilon\alpha L^{-R}$ therefore bracket $n-h_m-w$ for every
fixed $\varepsilon>0$ and all large $m$.  Since
$\dd\E_z^{>}Y/\dd z=-\operatorname{Var}_z^{>}(Y)<0$, there is
\[
 z_m=z_0+o(\alpha L^{-R})
\]
with $\E_{z_m}^{>}Y=n-h_m-w$.  By \eqref{eq:Yvar} and Chebyshev,
\[
 \Pp_{z_m}^{>}(n-h_m-2w\le Y\le n-h_m)=1-o(1).
\]
Let
\[
 A_m=\{n-h_m-2w\le Y\le n-h_m\}.
\]
The preceding Chebyshev estimate says $\Pp_{z_m}^{>}(A_m)=1-o(1)$.
Apply Lemma~\ref{lem:window-transfer}.  On $A_m$ the completed total lies in
$[n-h_m-2w,n]$, and therefore, with $W_m=h_m+2w$,
\begin{equation}\label{eq:window}
 \Pp_{z_m}(n-W_m\le T\le n)
 \ge e^{-C h_m}(1-o(1))
 =\exp\{-O(h_m)\}.
\end{equation}
This is the only loss caused by inserting the small-prime controls.

For every $z>0$, the $n$th term of $G_m$ gives
$s(n,m)e^{-zn}\le G_m(z)$.  Using \eqref{eq:G-exp} at $z_0$ gives the required upper
bound.  For the lower bound, apply Lemma~\ref{lem:coefficient-window} to
\eqref{eq:window}.  Since its probability loss is $H=O(h_m)$, we get
\[
 \log s(n,m)\ge \log G_m(z_m)+z_mn-z_mW_m-O(h_m+\log(W_m+1)).
\]
Here $z_mW_m=O(m^{2/3}+h_m)=o(m/L^R)$.  It remains only to justify replacing
$z_m$ by $z_0$ in the saddle exponent.  For
$\Phi_m(z)=\log G_m(z)+zn$ one has
$\Phi_m'(z)=n-\E_zT$, and for $z\asymp\alpha$ the elementary product model gives
$|\Phi_m'(z)|\ll n$.  Since
$z_m-z_0=o(\alpha L^{-R})$ and $n=m/\alpha$, the mean-value theorem yields
\[
 \Phi_m(z_m)-\Phi_m(z_0)=o(m/L^R).
\]
It remains to evaluate the saddle exponent explicitly.  Write
$W=W_{\alpha,R}(1/L)$, so that $z_0=\alpha W$ and $n=m/\alpha$.  From
\eqref{eq:G-exp} and Stirling's formula,
\[
\begin{aligned}
 \log G_m(z_0)+z_0n
 &=m\log n-\psi(m)
   +m\left(W-\log W-1
   +\sum_{r=1}^{R}\frac{A_r(\alpha W)}{L^r}\right)
   +o(m/L^R).
\end{aligned}
\]
The finite-order identity \eqref{eq:saddle-value-residual} makes the
parenthesis equal to $\sum_{r=1}^{R}E_r(\alpha)L^{-r}+O_R(L^{-R-1})$.
Together with the upper and lower bounds above, this proves \eqref{eq:bulk}.
\end{proof}

For the complementary range no saddle is needed.

\begin{proposition}[Far-sparse range]\label{prop:sparse}
Uniformly for $n>m^2$,
\begin{equation}\label{eq:sparse-log}
 \log s(n,m)=m\log n-\psi(m)+\log\frac mn+O\!\left(\frac{m^2}{n}\right).
\end{equation}
\end{proposition}

\begin{proof}
From Proposition~\ref{prop:binomial}, uniformly for $0\le h<m$,
\[
 \binom{n-h+m-1}{m-1}
 =\frac{n^{m-1}}{(m-1)!}
 \prod_{r=1}^{m-1}\left(1+\frac{r-h}{n}\right)
 =\frac{n^{m-1}}{(m-1)!}\exp\!\left\{O\!\left(\frac{m^2}{n}\right)\right\}.
\]
Using \eqref{eq:cm-total} gives
\[
 s(n,m)=\frac{m}{L_m}n^{m-1}
 \exp\!\left\{O\!\left(\frac{m^2}{n}\right)\right\},
\]
which is \eqref{eq:sparse-log}.
\end{proof}

The two estimates have a wide enough overlap for the global error scale.
Indeed, when $n\ge m^2$ one has $\alpha=m/n\le1/m$ and
$E_r(\alpha)=O_r(\alpha)$, so
\[
 m\frac{E_r(\alpha)}{(\log m)^r}=O_r((\log m)^{-r}).
\]
At the same time the far-sparse error is $O(m^2/n)=O(1)$.  Both quantities
are $o(m/(\log m)^R)$ for every fixed $R$.  Thus no transition region is
lost when the two arguments are joined.

\begin{proof}[Proof of Theorem~\ref{thm:main}]
For $m\le n\le m^2$, Proposition~\ref{prop:bulk} applies; in this range
$|\log(m/n)|\le\log m=o(m/(\log m)^R)$, so the explicit logarithmic term may be inserted.
For $n>m^2$, Proposition~\ref{prop:sparse} applies.  Since $\alpha=m/n<1/m$ and
$E_r(\alpha)=O_r(\alpha)$,
\[
 m\frac{E_r(\alpha)}{(\log m)^r}=O_r((\log m)^{-r}),
\]
while $m^2/n\le1=o(m/(\log m)^R)$.  The two ranges give \eqref{eq:main-psi} uniformly.
Equation \eqref{eq:main-simple} follows from the standard estimate for $\psi(m)$.
\end{proof}

\section{Sparse asymptotics}\label{sec:transition}

The main theorem is logarithmic and therefore does not by itself show exactly when the
leading term $(m/L_m)n^{m-1}$ becomes multiplicatively accurate as $m$ grows.  The exact
binomial identity from Section~\ref{sec:residue} contains finer information.  We use it
here to identify the full hierarchy of multiplicative correction scales.

Let $H$ be the residue sum under the uniform law on admissible residue vectors:
\[
 \Pp(H=h)=\frac{c_m(h)}{m!/L_m}.
\]
Then Proposition~\ref{prop:binomial} gives
\begin{equation}\label{eq:ratio-expectation}
 \frac{s(n,m)}{(m/L_m)n^{m-1}}
 =\E\prod_{r=1}^{m-1}\left(1+\frac{r-H}{n}\right).
\end{equation}
The following cumulant estimate is the only additional input.

\begin{lemma}[Derivative of the control correction at the origin]\label{lem:control-derivative}
Let
\[
 \mathcal R_m(z)=\log P_m(z)+\psi(m)
 -\sum_{\sqrt m<p\le m}\log F_{p,M_p}(z).
\]
Then
\[
 \mathcal R_m(0)=0,\qquad \mathcal R_m'(0)=O(h_m).
\]
The implied constant is absolute.
\end{lemma}

\begin{proof}
Condition first on all large-prime events and write $X$ for the remaining
noncontrol coordinates.  If $r^*_{Q_p}(X)$ is the unique completing control,
set
\[
 g_{Q,z}(r)=Q\mu_{Q,z}(r)
 =Q\frac{1-e^{-z/Q}}{1-e^{-z}}e^{-zr/Q},
 \qquad
 \widetilde W_z(X)=\prod_{p\le\sqrt m}
 g_{Q_p,z}(r^*_{Q_p}(X)).
\]
The factorization in Lemma~\ref{lem:controls} gives
\[
 \mathcal R_m(z)=\log \E_z^{>}\widetilde W_z(X).
\]
At $z=0$, $g_{Q,0}(r)=1$ for every $Q,r$, hence
$\widetilde W_0\equiv1$ and $\mathcal R_m(0)=0$.  When differentiating the
expectation at $0$, the derivative of the probability measure contributes
zero because it is applied to the constant function $1$.  Thus
\[
 \mathcal R_m'(0)=\E_0^{>}\!\left[
 \partial_z\widetilde W_z(X)\big|_{z=0}\right].
\]
A direct expansion gives
\[
 \partial_z\log g_{Q,z}(r)\big|_{z=0}
 =\frac12-\frac{r+1/2}{Q},
\]
whose absolute value is at most $1/2$.  Since there are $h_m$ controls,
$|\partial_z\widetilde W_z|_{z=0}|\le h_m/2$.
\end{proof}

\begin{lemma}[Small-tilt residue cumulant]\label{lem:cumulant}
Fix an integer $R\ge1$ and $C>0$.  Define
\[
 K_m(z)=\log\E e^{z(m/2-H)}.
\]
Uniformly for $0\le z\le C(\log m)^R/m$,
\begin{equation}\label{eq:Kexp}
 K_m(z)=zm\sum_{r=1}^{R}\frac{\kappa_r}{(\log m)^r}+o(1).
\end{equation}
\end{lemma}

\begin{proof}
Let $C_m(z)=\sum_hc_m(h)e^{-zh}$.  Then
$C_m(z)=P_m(z)\prod_{j\le m}S_j(z)$ and
\[
 K_m(z)=\frac{zm}{2}+\log C_m(z)-\log C_m(0).
\]
At $z=0$, with $H_m=\sum_{j\le m}j^{-1}$,
\[
 \sum_{j\le m}(\log S_j)'(0)=-\frac m2+\frac{H_m}{2}.
\]
For the large-prime part, Theorem~\ref{thm:prime-block}, with derivative order
one and the same fixed $R$, gives
\[
 \sum_{\sqrt m<p\le m}(\log F_{p,M_p})'(0)
 =m\sum_{r=1}^{R}\frac{A_r'(0)}{(\log m)^r}
 +o\!\left(\frac{m}{(\log m)^R}\right).
\]
Lemma~\ref{lem:control-derivative} shows that the entire small-prime control
correction contributes only $O(h_m)$ to this derivative.  Now
$(\log F_1)'(0)=1/2$, whereas $F_M'(0)=0$ for $M\ge2$.  Hence only the
$M=1$ term in \eqref{eq:Ar} contributes to $A_r'(0)$.  Since
\[
 a_{r,1}=(r-1)!\left(1-\frac12
 \sum_{j=0}^{r-1}\frac{(\log 2)^j}{j!}\right),
\]
we have $A_r'(0)=a_{r,1}/2=\kappa_r$.  As
$H_m=O(\log m)$ and $h_m=o(m/(\log m)^R)$ for every fixed $R$, hence
\begin{equation}\label{eq:Kprime-all-orders}
 K_m'(0)=m\sum_{r=1}^{R}\frac{\kappa_r}{(\log m)^r}
 +o\!\left(\frac{m}{(\log m)^R}\right).
\end{equation}

It remains to bound $K_m''(z)=\operatorname{Var}_z(H)$.  First omit the
small-prime controls and condition only on the disjoint large-prime blocks.
The same second-derivative identity used in Lemma~\ref{lem:mean-var} shows
that each conditioned block has variance $O(M_p)$; residue coordinates not
belonging to such a block remain independent and each has variance at most
$1/4$.  Hence the noncontrol residue sum $H_0$ satisfies
\[
 \operatorname{Var}_z^{>}(H_0)=O(m)
\]
uniformly for $0\le z\le C(\log m)^R/m$, since this range is eventually
contained in any fixed bounded $z$-interval.

Now restore the controls.  Their total contribution $C(X)$ lies in
$[0,h_m)$, so $\operatorname{Var}(C(X))\le h_m^2=O(m)$.  Also, uniformly in
$Q$ and $r$,
\[
 Q\mu_{Q,z}(r)=\exp\{O(z)\}.
\]
For $z\le C(\log m)^R/m$ one has $h_mz=o(1)$, so the product of the $h_m$
control weights is bounded above and below by positive constants for all
sufficiently large $m$.  After normalization, the Radon--Nikodym derivative
of the completed law with respect to the noncontrol law is therefore bounded
above and below by constants depending only on $R$ and $C$.  If
$a=\E_z^{>}H_0$, then
\[
 \operatorname{Var}_z^{\rm all}(H_0)
 \le \E_z^{\rm all}(H_0-a)^2
 \ll_{R,C}\E_z^{>}(H_0-a)^2\ll_{R,C} m.
\]
Finally $H=H_0+C(X)$, so Cauchy--Schwarz gives
$\operatorname{Var}_z^{\rm all}(H)=O_{R,C}(m)$.  Thus
$K_m''(z)=O_{R,C}(m)$ uniformly throughout the stated range.  Taylor's
theorem and \eqref{eq:Kprime-all-orders} give
\[
 K_m(z)=zK_m'(0)+O_{R,C}(mz^2).
\]
Here $z\,o(m/(\log m)^R)=o(1)$ and
$mz^2=O_{R,C}((\log m)^{2R}/m)=o(1)$, proving \eqref{eq:Kexp}.
\end{proof}

\begin{proof}[Proof of Theorem~\ref{thm:transition}]
Put $L=\log m$ and $z_0=(m-1)/n$.  Uniformly for $0\le H<m$,
\[
 \log\prod_{r=1}^{m-1}\left(1+\frac{r-H}{n}\right)
 =z_0\left(\frac m2-H\right)
 +O\!\left(\frac{m^3}{n^2}\right).
\]
If $n\ge m^2/(CL^R)$, then $z_0\le CL^R/m$ and
$m^3/n^2=O_{R,C}(L^{2R}/m)=o(1)$.  Thus there is a uniform
$\varepsilon_m=o(1)$ such that
\[
 e^{-\varepsilon_m}e^{z_0(m/2-H)}
 \le \prod_{r=1}^{m-1}\left(1+\frac{r-H}{n}\right)
 \le e^{\varepsilon_m}e^{z_0(m/2-H)}.
\]
Taking expectations and logarithms, \eqref{eq:ratio-expectation} and
Lemma~\ref{lem:cumulant} give, uniformly in this range,
\[
 \log\frac{s(n,m)}{(m/L_m)n^{m-1}}
 =z_0m\sum_{r=1}^{R}\frac{\kappa_r}{L^r}+o(1).
\]
Replacing $z_0=(m-1)/n$ by $m/n$ changes the right side by
$O_{R,C}(L^{R-1}/m)=o(1)$, proving \eqref{eq:sparse-all-orders}.

Taking $R=1$ gives \eqref{eq:transition} whenever
$nL/m^2\to c\in(0,\infty]$.  Taking $R=2$ and
$n=\lfloor c m^2/L^2\rfloor$ gives
\[
 \frac{m^2}{n}\left(\frac{\kappa_1}{L}+\frac{\kappa_2}{L^2}\right)
 =\frac{L}{4c}+\frac{1-\log 2}{4c}+o(1),
\]
which yields \eqref{eq:second-sparse-scale}.  Finally, for general fixed $R$
and $n=\lfloor c m^2/L^R\rfloor$, equation
\eqref{eq:sparse-all-orders} becomes
\[
 \log\frac{s(n,m)}{(m/L_m)n^{m-1}}
 =\frac1c\sum_{r=1}^{R}\kappa_rL^{R-r}+o(1).
\]
Exponentiating proves \eqref{eq:R-sparse-scale}.
\end{proof}

\section{Numerical evaluation of the first coefficients}\label{sec:numerics}

The decimal values in \eqref{eq:diag-values} are included only to give a
sense of size.  They can be checked without summing the slowly convergent
series in \eqref{eq:FM} directly.  For $M\ge2$ and $D=D_M$, the classical
partial-fraction expansion of the hyperbolic cotangent gives
\begin{equation}\label{eq:coth-check}
 \sum_{\ell\in\mathbb Z}\frac1{(z+2\pi iD\ell)^M}
 =\frac{(-1)^{M-1}}{(M-1)!}
 \frac{d^{M-1}}{dz^{M-1}}
 \left(\frac1{2D}\coth\frac{z}{2D}\right).
\end{equation}
Thus $F_M(z)$ and the derivatives needed in
\eqref{eq:E1}--\eqref{eq:E3} reduce to derivatives of elementary hyperbolic
functions.  Evaluating the terms $M\le6$ at $z=1$ gives
\[
\begin{aligned}
 E_1(1)&=0.225872291710\ldots,\\
 E_2(1)&=0.046894482099\ldots,\\
 E_3(1)&=0.017888501642\ldots.
\end{aligned}
\]
The remaining $M$-tail is tiny.  Indeed $D_M\ge D_7=420$ for $M\ge7$, and
from \eqref{eq:FM}
\[
 |F_M(1)-1|
 \le 2\sum_{\ell\ge1}(2\pi D_M\ell)^{-M}
 \le \frac{2\zeta(M)}{(840\pi)^M}.
\]
Direct differentiation gives the same bound with only a fixed polynomial
factor in $M$ for the first two derivatives.  This is sufficient for
the nine decimal places in \eqref{eq:diag-values}.  No numerical
estimate enters any theorem.


\begin{thebibliography}{99}

\bibitem{BeckRobins15}
M.~Beck and S.~Robins,
\emph{Computing the Continuous Discretely}, 2nd ed., Springer, New York, 2015.

\bibitem{DilcherVignat17}
K.~Dilcher and C.~Vignat,
\emph{An explicit form of the polynomial part of a restricted partition function},
Res. Number Theory \textbf{3} (2017), Art.~1, 12 pp.
\href{https://doi.org/10.1007/s40993-016-0065-3}{doi:10.1007/s40993-016-0065-3}.

\bibitem{JiangWang19}
T.~Jiang and K.~Wang,
\emph{A generalized Hardy--Ramanujan formula for the number of restricted integer partitions},
J. Number Theory \textbf{201} (2019), 322--353.
\href{https://doi.org/10.1016/j.jnt.2019.02.006}{doi:10.1016/j.jnt.2019.02.006}.

\bibitem{KimKim25}
B.~Kim and E.~Kim,
\emph{Distributions of reciprocal sums of parts in integer partitions},
J. Combin. Theory Ser. A \textbf{211} (2025), Art.~105982.
\href{https://doi.org/10.1016/j.jcta.2024.105982}{doi:10.1016/j.jcta.2024.105982}.

\bibitem{Kiran22}
N.~Uday Kiran,
\emph{An algebraic approach to $q$-partial fractions and Sylvester denumerants},
Ramanujan J. \textbf{59} (2022), 671--712.
\href{https://doi.org/10.1007/s11139-022-00595-z}{doi:10.1007/s11139-022-00595-z}.

\bibitem{MontgomeryVaughan07}
H.~L.~Montgomery and R.~C.~Vaughan,
\emph{Multiplicative Number Theory I. Classical Theory},
Cambridge Univ. Press, Cambridge, 2007.

\bibitem{OEIS208480}
OEIS Foundation Inc.,
\emph{A208480: Number of representations of $n$ as $m_1/1+\cdots+m_n/n$},
\url{https://oeis.org/A208480}.

\bibitem{OSullivan16}
C.~O'Sullivan,
\emph{Asymptotics for the partial fractions of the restricted partition generating function, I},
Int. J. Number Theory \textbf{12} (2016), no.~6, 1421--1474.
\href{https://doi.org/10.1142/S1793042116500895}{doi:10.1142/S1793042116500895}.

\bibitem{OSullivan18}
C.~O'Sullivan,
\emph{Partitions and Sylvester waves},
Ramanujan J. \textbf{47} (2018), no.~2, 339--381.
\href{https://doi.org/10.1007/s11139-017-9939-9}{doi:10.1007/s11139-017-9939-9}.

\bibitem{OSullivan20}
C.~O'Sullivan,
\emph{Rademacher's conjecture and expansions at roots of unity of products generating restricted partitions},
J. Number Theory \textbf{216} (2020), 335--379.
\href{https://doi.org/10.1016/j.jnt.2020.04.005}{doi:10.1016/j.jnt.2020.04.005}.

\bibitem{SinhaMO12}
N.~K.~Sinha,
\emph{A simple looking problem in partitions that became increasingly complex},
MathOverflow, Question 96204 (2012),
\url{https://mathoverflow.net/questions/96204}.

\bibitem{Sylvester1857}
J.~J.~Sylvester,
\emph{On the partition of numbers},
Quart. J. Pure Appl. Math. \textbf{1} (1857), 141--152.

\bibitem{XinZhang25}
G.~Xin and C.~Zhang,
\emph{An algebraic combinatorial approach to Sylvester's denumerant},
Ramanujan J. \textbf{66} (2025), Art.~64, 28 pp.
\href{https://doi.org/10.1007/s11139-024-01022-1}{doi:10.1007/s11139-024-01022-1}.

\end{thebibliography}
\end{document}